\documentclass[12pt]{article}
\usepackage{amsmath}
\usepackage{latexsym}
\usepackage{amssymb}
\newtheorem{thm}{Theorem}[section]
\newtheorem{la}[thm]{Lemma}
\newtheorem{Defn}[thm]{Definition}
\newtheorem{Exa}[thm]{Example}
\newtheorem{Remark}[thm]{Remark}
\newtheorem{Problem}[thm]{Problem}
\newtheorem{prop}[thm]{Proposition}
\newtheorem{cor}[thm]{Corollary}
\newtheorem{Number}[thm]{\!\!}

\newenvironment{exa}{\begin{Exa}\rm}{\end{Exa}}

\newenvironment{numba}{\begin{Number}\rm}{\end{Number}}
\newenvironment{proof}{{\noindent\bf Proof.}}%
                  {\nopagebreak\hspace*{\fill}$\Box$\medskip\par}
\newcommand{\Punkt}{\nopagebreak\hspace*{\fill}$\Box$}

\newcommand{\ve}{\varepsilon}

\newcommand{\mto}{\mapsto}
\newcommand{\N}{{\mathbb N}}
\newcommand{\R}{{\mathbb R}}
\newcommand{\C}{{\mathbb C}}

\newcommand{\K}{{\mathbb K}}

\newcommand{\cg}{{\mathfrak g}}

\newcommand{\sub}{\subseteq}

\DeclareMathOperator{\id}{id}
\newcommand{\sbull}{{\scriptscriptstyle \bullet}}
\DeclareMathOperator{\Diff}{Diff}

\DeclareMathOperator{\Gh}{Gh}
\DeclareMathOperator{\Germ}{Germ}

\DeclareMathOperator{\GermDiff}{GermDiff}
\DeclareMathOperator{\Homeo}{Homeo}
\DeclareMathOperator{\graph}{graph}
\newcommand{\pl}{{\displaystyle \lim_{\longleftarrow}\, }}
\newcommand{\dl}{{\displaystyle \lim_{\longrightarrow}\, }}
\begin{document}
$\,$\\[-29mm]
\begin{center}
{\Large\bf
Completeness of locally {\boldmath$k_\omega$}-groups and\\[2mm]
related infinite-dimensional Lie groups}\\[5mm]
{\bf Helge Gl\"{o}ckner}\vspace{2mm}
\end{center}
\begin{abstract}
\hspace*{-6mm}Recall that a topological space $X$ is said to be a \emph{$k_\omega$-space}
if it is the direct limit of an ascending sequence $K_1\sub K_2\sub\cdots$
of compact Hausdorff topological spaces. If each point in a Hausdorff space $X$
has an open neighbourhood which is a $k_\omega$-space, then $X$ is called
\emph{locally $k_\omega$}. We show that a topological group is complete whenever the
underlying topological space is locally $k_\omega$.
As a consequence, every infinite-dimensional Lie group
modelled on a Silva space is complete.
\end{abstract}
{\bf Classification:}
22E65 (primary);
22A05, 
46A13, 
46M40, 
58D05 
\\[2.3mm]
\section{Introduction and statement of the results}
Recall that a \emph{$k_\omega$-space} is a Hausdorff topological
space~$X$ which carries the direct limit topology for
an ascending sequence $K_1\sub K_2\sub\cdots$
of compact subsets $K_n\sub X$ with $\bigcup_{n\in\N}K_n=X$
(see \cite{FaT}, \cite{GGH}; cf.\ \cite{MiE} and, with different terminology, \cite{Mor}).
A topological group is called~a \emph{$k_\omega$-group}
if its underlying topological space is a $k_\omega$-space.
Hunt and Morris~\cite{HaM} showed that
every $k_\omega$-group is Weil complete,
viz., complete in its left uniformity
(cf.\ \cite{Rai} for the case of abelain $k_\omega$-groups;
see also \cite{Bra} for a recent proof).
The current work is devoted to generalizations and applications
of this fact,
with a view towards examples
in infinite-dimensional Lie theory.

Following \cite{GGH}, a Hausdorff space~$X$
is called \emph{locally $k_\omega$} if each $x\in X$
has an open neighbourhood $U\sub X$ which
is a $k_\omega$-space in the induced topology.
A topological group~$G$ is called locally~$k_\omega$
if its underlying topological space is locally~$k_\omega$.
Since every locally $k_\omega$ group has an open subgroup
which is a $k_\omega$-group \cite[Proposition 5.3]{GGH},
the Hunt-Morris Theorem implies the following:
\begin{prop}\label{locally}
Every locally $k_\omega$ topological group 
is Weil complete.\,\Punkt
\end{prop}
We consider
Lie groups modelled on arbitrary real or complex Hausdorff locally convex topological
vector spaces (as in \cite{Nee} and \cite{GaN}),\footnote{Compare also
\cite{Mil} (for Lie groups modelled
on sequentially complete spaces).}
based on the differential calculus in locally convex spaces
known as Keller's $C^\infty_c$-theory~\cite{Kel}.
Since the smooth maps under consideration are, in particular,
continuous, the Lie groups we consider have continuous group operations.
They can
therefore
be regarded as Hausdorff topological groups,
and we can ask when they are complete.
The
following
observation (proved in Section~\ref{secmod}) is essential.
\begin{prop}\label{modelloc}
If a Lie group $G$ is modelled on a locally convex space~$E$ which is a $k_\omega$-space,
then $G$ is locally $k_\omega$, and its identity component~$G_e$ is~$k_\omega$.
\end{prop}
Propositions \ref{locally} and \ref{modelloc}
entail a simple completeness criterion:
\begin{cor}\label{ifmodel}
Every Lie group modelled on a $k_\omega$-space is Weil complete.\,\Punkt
\end{cor}
Recall that a locally convex space $E$ is called a \emph{Silva space}
if it is the locally convex direct limit $E=\dl E_n$\vspace{-.7mm} of an ascending
sequence
$E_1\sub E_2\sub\cdots$
of Banach spaces, such that each inclusion map $E_n\to E_{n+1}$
is a compact operator (cf.\ \cite{Flo}).
It is well known that every Silva space is a $k_\omega$-space
(see, e.g., \cite[Example 9.4]{JFA}).
Thus Corollary~\ref{ifmodel} entails:
\begin{cor}\label{Silva}
Lie groups modelled on Silva spaces are Weil
complete.\,\Punkt
\end{cor}
Until recently, little was known on the completeness properties
of infinite-dimensional Lie groups, except for the classical
fact that Lie groups modelled on Banach spaces are
Weil complete (see Proposition~1 in \cite[Chapter~III, \S1.1]{Bou}).
In 2016, Weil completeness was established
for many classes of infinite-dimensional Lie groups~\cite{Com},
but some examples modelled on Silva spaces could not be treated.
The current paper closes this gap, as Corollary~\ref{Silva}
establishes Weil completeness for the latter.
Section~\ref{secexa} compiles a list 
of infinite-dimensional Lie groups which are modelled
on Silva spaces and hence Weil complete (by Corollary~\ref{Silva}).
In particular, we find:
\begin{itemize}
\item
For each compact real analytic manifold~$M$, the Lie group
$\Diff^\omega(M)$
of all real analytic diffeomorphisms $\phi\colon M\to M$
is Weil complete;
\item
For each finite-dimensional Lie group~$G$
and $M$ as before,
the Lie group $C^\omega(M,G)$
of all real analytic maps $f\colon M\to G$
is Weil complete.
\end{itemize}
Being modelled on Silva spaces, the examples we consider
are locally $k_\omega$ (by Proposition~\ref{modelloc}),
which is sufficient to conclude Weil completeness.
Yet, it is natural to ask whether the groups in question
are not only locally $k_\omega$, but, actually,  $k_\omega$-groups.
Results in Section~\ref{seckomeg} subsume:
\begin{prop}\label{arekomeg}
For each compact real analytic manifold~$M$,
the Lie group $\Diff^\omega(M)$ is a $k_\omega$-group.
Moreover, $C^\omega(M,G)$ is a $k_\omega$-group
for each
$\sigma$-compact finite-dimensional Lie group~$G$,
and~$M$ as before.
\end{prop}
Our studies are related to an open problem by
K.-H. Neeb,
who asked whether every Lie group modelled on a complete locally convex space
is (Weil) complete (cf.\ \cite[Problem II.9]{Nee}).
As long as this problem remains open, it is natural to look for classes $\Omega$
of locally convex spaces such that every Lie group modelled on a space $E\in\Omega$
is Weil complete. Banach spaces form one such class~$\Omega$.
By Corollary~\ref{ifmodel}, locally convex spaces whose underlying topological
space is~$k_\omega$
furnish a second such class.\\[2.3mm]
\emph{Acknowledgement.}
The author is grateful to Hans Braun (Paderborn) for discussions
concerning Ra\u{\i}kov's work~\cite{Rai},
and to Alexander Schmeding (Trondheim)
for comments concerning the manifold structure on $\Diff^\omega(M)$.
\section{Proof of Proposition~\ref{modelloc}}\label{secmod}
For each $g\in G$, there exists a homeomorphism $\phi\colon U\to V$
from an open neighbourhood $U\sub G$ of~$g$ onto an open subset $V\sub E$.
By \cite[Proposition 4.2\,(g)]{GGH}, $V$ contains an open neighbourhood $W$ of~$\phi(g)$
which is a $k_\omega$-space in the induced topology.
Then $\phi^{-1}(W)$ is an open neighbourhood of~$g$ in~$G$
and a $k_\omega$-space. Thus $G$ is locally~$k_\omega$.
Hence~$G$ has an open subgroup~$S$ which is~$k_\omega$
and since $G_e$ is a closed subgroup of~$S$, it is~$k_\omega$ as well
(see \cite[Propositions 5.3 and 5.2\,(b)]{GGH}).\,\Punkt
\section{Examples of Weil complete Lie groups}\label{secexa}
Corollary~\ref{Silva} applies to many classes of infinite-dimensional Lie groups.
\begin{exa}\label{diffana}
If $M$ is a compact real analytic manifold (without boundary or corners),
then the group $\Diff^\omega(M)$ of all real analytic diffeomorphisms
$\phi\colon M\to M$
is a Lie group modelled on the Silva space $\Gamma^\omega(TM)$ of all real analytic vector fields on~$M$
(see \cite{Les}, \cite{KaM}, \cite{DaS}).
By Corollary~\ref{Silva}, $\Diff^\omega(M)$ is Weil complete.
\end{exa}
\begin{exa}\label{exgerm}
Let $G$ be a finite-dimensional Lie group over $\K\in\{\R,\C\}$,
with modelling space~$\cg$,
and~$K$ be a non-empty compact subset of a finite-dimensional
$\K$-vector space~$E$.
Given a $\K$-analytic map $f\colon U\to G$ on an open subset $U\sub E$ with $K\sub U$,
write $[f]$ for its germ around~$K$.
Then the group
$\Germ(K,G)$ of all such germs $[f]$
is a Lie group modelled on the Silva space $\Germ(K,\cg)$~(see \cite{GER});
likewise if $K\not= \emptyset$ is a compact subset of a finite-dimensional
$\K$-analytic manifold~$M$ (see~\cite{DGS}).
By Corollary~\ref{Silva}, $\Germ(K,G)$ is Weil complete.
\end{exa}
\begin{exa}\label{anacomp}
If $M$ is a compact real analytic manifold, taking $\K:=\R$ and $K:=M$ in Example~\ref{exgerm}
we see that the Lie group $C^\omega(M,G)$
of $G$-valued real analytic maps on~$M$ (which can be identified with $\Germ(M,G)$)
is Weil complete, for each finite-dimensional Lie group~$G$.
\end{exa}
\begin{exa}\label{anaR}
If $G$ is a Lie group modelled on a finite-dimensional real vector space~$\cg$, then the
group $C^\omega(\R,G)$ of all real analytic mappings\linebreak
$f\colon\R\to G$
can be made a Lie group modelled on the projective limit
$C^\omega(\R,\cg)$ $=\pl \Germ([{-n},n],\cg)$\vspace{-.7mm}
of Silva spaces,
using the strategy of~\cite{NaW} (cf.\ \cite{DGS}).
As $\Germ([{-n},n],G)$ is a Weil complete Hausdorff group for each $n\in\N$ by Example~\ref{anacomp}
and $C^\omega(\R,G)=\pl \!\Germ([{-n},n],G)$,\vspace{-.7mm} we see that
$C^\omega(\R,G)$ is Weil complete.
\end{exa}
\begin{exa}\label{germdiff}
For $\K\in\{\R,\C\}$ and a non-empty compact subset~$K$
of a finite-dimensional $\K$-vector space~$E$, consider the group
$\GermDiff(K)$ of germs~$[f]$ around~$K$ of $\K$-analytic diffeomorphisms
$f\colon U\to V$ between open subsets $U,V\sub E$ with $K\sub U\cap V$, such that $f|_K=\id_K$.
Then $\GermDiff(K)$ is a Lie group modelled on the Silva space $\Germ(K,E)_0$
of all germs of $E$-valued $\K$-analytic maps around $K\sub E$
which vanish on~$K$ (see \cite{JFA}). By Corollary~\ref{Silva},
$\GermDiff(K)$ is Weil complete.
As a special case, taking $\K:=\C$, $E:=\C^n$ and $K:=\{0\}$,
we see that the Lie groups $\Gh_n(\C)=\GermDiff(\{0\})$
studied by Pisanelli~\cite{Pis} are Weil complete.
We also mention that certain Lie groups of real
analytic diffeomorphisms considered by Leitenberger (see~\cite{Lei})
are closed subgroups of $\GermDiff(\{0\})$
with $\K:=\R$, $E:=\R^n$, and $K:=\{0\}$,
whence they are Weil complete as well.
\end{exa}
\begin{exa}
The tame Butcher group over $\K\in\{\R,\C\}$ (see \cite{BfS})
is an infinite-dimensional Lie group related to
integration methods for
ordinary differential equations in numerical analysis.
As the tame Butcher group is globally diffeomorphic
to a Silva space (see \cite[Lemma~1.17]{BfS}), it is a $k_\omega$-group
and hence Weil complete, by the Hunt-Morris Theorem.
\end{exa}
\begin{exa}
Let $M$ be a compact smooth manifold, $G$ be a Lie group modelled on a finite-dimensional real vector space~$\cg$, and $s$ be a real number such that $s\geq \dim(M)/2$.
Then a Lie group $H^{>s}(M,G)$ can be defined
which is modelled on the Silva space
$\dl H^{s+\frac{1}{n}}(M,\cg)$\vspace{-.4mm}
of all $\cg$-valued functions on~$M$ which are of Sobolev class $t$ for some $t>s$
(see \cite{GaT}). By Lemma~\ref{Silva}, $H^{>s}(M,G)$
is Weil complete.
\end{exa}
\begin{exa}\label{outslie}
Consider an ascending sequence $G_1\sub G_2\sub\cdots$
of finite-dimensional real Lie groups
such that all inclusion maps $G_n\to G_{n+1}$
are smooth group homomorphisms.
Then $G:=\bigcup_{n\in\N}G_n$
can be made a Lie group modelled on the Silva space
$\dl\,L(G_n)$\vspace{-.3mm} (see \cite{DL},
or also \cite{NRW}, \cite{KaM} for special cases).
By Corollary~\ref{Silva}, $G$ is Weil complete.
\end{exa}
\section{Proof of Proposition~\ref{arekomeg}}\label{seckomeg}
In this section, we prove Proposition~\ref{arekomeg}
and obtain results concerning
separabilty and the $k_\omega$-property
also for some further examples of Lie groups.

Recall that a topological space is called \emph{separable}
if it has a countable, dense subset.
The following observation will help us to prove Proposition~\ref{arekomeg}:
\begin{la}\label{denso}
If a topological group $G$ is separable and locally $k_\omega$,
then $G$ is a $k_\omega$-group.
\end{la}
\begin{proof}
Let $D\sub G$ be a dense countable subset.
Since~$G$ is locally~$k_\omega$, it has an open subgroup~$U$ which is a $k_ \omega$-group
(see \cite[Proposition~5.3]{GGH}).
Then $G=DU$, enabling us to choose a subset $R\sub D$ of representatives
for the left cosets of~$U$.
Thus $G=RU$ and $rU\cap sU=\emptyset$ for all $r\not=s$ in~$R$.
As each coset is open, the disjoint union
$G=\bigcup_{r\in R}rU$
is a topological sum. Since countable topological sums of $k_\omega$-spaces
are $k_\omega$ (see \cite[Proposition 4.2\,(e)]{GGH}), we see that $G$ is a $k_\omega$-space.
\end{proof}
\begin{numba}\label{baby-metric}
Every subset of a separable metric space is separable, as is well known.
\end{numba}
The next lemma compiles further well-known elementary facts.
\begin{la}\label{easysep}
\begin{itemize}
\item[\rm(a)]
If a topological space~$X$ is separable,
then every open subset of~$X$ is separable.
\item[\rm(b)]
If $(X_j)_{j\in J}$ is a family of separable topological spaces
with countable index set $J\not=\emptyset$, then $\prod_{j\in J}X_j$
is separable when endowed with the product topology.
\item[\rm(c)]
If $(X_j)_{j\in J}$ is a family of separable topological spaces
with countable index set $J\not=\emptyset$, then also the topological
sum $\coprod_{j\in J}X_j$
is separable.
\item[\rm(d)]
If $f\colon X\to Y$ is a surjective continuous map between topological
spaces and $X$ is separable, then $Y$ is separable.
\item[\rm(e)]
If $X$ is a topological space which has a separable dense
subset, then $X$ is separable.
\item[\rm(f)]
Let $X$ be a topological space, $(X_j)_{j\in J}$ be a family of separable topological spaces
with countable index set $J\not=\emptyset$, and $(f_j)_{j\in J}$
be a family of continuous maps $f_j\colon X_j\to X$ with
$\bigcup_{j\in J}f_j(X_j)=X$. Then $X$ is separable.\,\Punkt
\end{itemize}
\end{la}
If $X$ and $Y$ Hausdorff topological spaces,
we write $C(X,Y)$ for the set of all continuous maps $f\colon X\to Y$.
We shall always endow $C(X,Y)$ with the compact-open topology
(as discussed in \cite{Eng}, \cite{Str}, or also \cite[Appendix~A.5]{GaN}).
If $K$ is a compact Hausdorff topological space,
we write $\Homeo(K)$ for the group of all homeomorphisms
$\phi\colon K\to K$.
We endow $\Homeo(K)$ with the topology
induced by $C(K,K)$.
Some folklore facts concerning
the compact-open topology
will be used, as
compiled in the next lemma.
\begin{la}\label{folklore}
\begin{itemize}
\item[\rm(a)]
For each compact subset $K\sub\R^m$ and
$n\in\N$,
the locally convex space $C(K,\R^n)$ is separable.
\item[\rm(b)]
For each $K$ as before and $\sigma$-compact finite-dimensional Lie group~$G$,
the topological group $C(K,G)$ is separable.
\item[\rm(c)]
If $K$ is a compact smooth manifold $($or a compact subset of $\R^n$ for some $n\in\N)$,
then $\Homeo(K)$ is a separable topological group.
\end{itemize}
\end{la}
\begin{proof}
(a) By the Stone-Weierstra\ss{} Theorem,
the algebra
of
real-valued
polynomial functions
is dense in $C(K,\R)$, and hence also its dense countable
subset of polynomials with rational coefficients. Thus $C(K,\R^m)\cong C(K,\R)^m$
is separable.

(b) It is well known that $C(K,G)$ is a topological group (see, e.g., \cite[Theorem~11.5]{Str}).
By the Whitney Embedding Theorem (see, e.g., \cite{Hir}),
for some $n\in\N$ there exists a $C^\infty$-diffeomorphism
$j\colon G\to N$ from $G$ onto a $C^\infty$-submanifold $N\sub\R^n$.
Now
$C(K,G) \to C(K,\R^n)$, $f\mto j\circ f$
is a topological embedding. Hence $C(K,G)$ is separable, by
(a) and \ref{baby-metric}.

(c) By Whitney's Embedding Theorem, every compact smooth manifold
is $C^\infty$-diffeomorphic to a smooth submanifold of some $\R^n$
(see, e.g., \cite{Hir}).
It therefore suffices to consider a compact subset $K\sub\R^n$.
As the metrizable space $C(K,\R^n)$ is separable by~(a), also its subset $\Homeo(K)$
is separable, by \ref{baby-metric}.
It is well known that $\Homeo(K)$ is a topological group:
See, e.g., \cite[Lemma 9.4\,(c)]{Str}
for continuity of multiplication.
As the evaluation map
$\ve\colon \Homeo(K)\times K\to K$, $(\phi,x)\mto \phi(x)$
is continuous (see \cite[Lemma 9.8]{Str}), it has closed graph,
whence also
$h\colon \Homeo(K)\times K\to K$, $(\phi,x)\mto \phi^{-1}(x)$
has closed graph.\footnote{Note that $\graph(h)$ is obtained from $\graph(\ve) \sub \Homeo(K)\times K\times K$
by swapping the last and penultimate components.}
Since $K$ is compact, continuity of~$h$ follows (see, e.g., \cite[Theorem 1.21\,(b)]{Str})
and thus also continuity of
$h^\vee\colon \Homeo(K)\to \Homeo(K)$, $\psi\mto h(\psi,\sbull)=\psi^{-1}$
(cf., e.g., \cite[Theorem 3.4.1 and p.\,110]{Eng}).
\end{proof}
We shall recognize $k_\omega$-groups using the following lemma:
\begin{la}\label{givesko}
Let $G$ be a topological group. Assume that there exists a continuous homomorphism
$\alpha\colon G\to H$ to a separable metrizable topological group~$H$
and an open identity neighbourhood $U\sub H$ such that
$\alpha^{-1}(U)$ is separable in the topology induced by~$G$.
Then~$G$ is separable.
\end{la}
\begin{proof}
Since $H$ is separable and metrizable, $\alpha(G)$ is separable in the topology induced by~$H$
(see \ref{baby-metric}). We therefore find a countable subset $C\sub G$ such that $\alpha(C)$ is dense in~$\alpha(G)$.
Let $D$ be a dense countable subset of $\alpha^{-1}(U)$.
Then $C^{-1}D$ is a countable subset of~$G$.
For $g\in G$, we find $c\in C$ such that $\alpha(c)\in U\alpha(g^{-1})$
and thus $cg\in \alpha^{-1}(U)$. Now $d_j\to cg$ for a net $(d_j)_{j\in J}$ in $D$. Then $c^{-1}d_j\in C^{-1}D$ and $c^{-1}d_j\to g$,
whence $C^{-1}D$ is dense in~$G$.
\end{proof}
The next lemma yields separability of relevant modelling spaces.
\begin{la}\label{themodsep}
Every Silva space is separable.
\end{la}
\begin{proof}
Let $E_1\sub E_2\sub\cdots$ be an ascending sequence of Banach spaces,
such that all inclusion maps $E_n\to E_{n+1}$ are compact operators.
Consider the locally convex direct limit $E:=\bigcup_{n\in\N}E_n$.
Let~$B_n$ be the unit ball in~$E_n$ and~$K_n$ be its closure in~$E_{n+1}$.
Then $K_n$ is compact and metrizable as a subset of $E_{n+1}$
(hence also in~$E$), and thus separable.
As a consequence, $E=\bigcup_{n,m\in\N} mK_n$ is separable,
being a countable union of subsets which are separable in the induced topology
(see Lemma~\ref{easysep}\,(f)).
\end{proof}
In the next proposition, $\K\in\{\R,\C\}$.
\begin{prop}\label{germvers}
Let $M$ be a finite-dimensional $\K$-analytic manifold,
$G$ be a finite-dimensional $\K$-analytic Lie group which is $\sigma$-compact
and $K\sub M$ be a compact, non-empty subset.
Then $\Germ(K,G)$ $($as in Example~{\rm\ref{exgerm})}
is separable and a $k_\omega$-group.
\end{prop}
\begin{proof}
Let $\phi\colon U\to V$ be a $\K$-analytic diffeomorphism
from an open identity neighbourhood $U\sub G$ onto an
open subset $V\sub \cg$ such that $\phi(e)=0$,
where $\cg$ ($\cong \K^n$) is the modelling space of~$G$.
Then
\[
\Germ(K,V):=\{[f]\in\Germ(K,\cg)\colon f(K)\sub V\}
\]
is an open $0$-neighbourhood in the Silva space~$\Germ(K,\cg)$
and hence separable, by Lemmas~\ref{themodsep} and \ref{easysep}\,(a).
After shrinking $U$ if necessary,
\[
\Germ(K,U):=\{[f]\in\Germ(K,G)\colon f(K)\sub U\}
\]
is an open identity neighbourhood and the map
$\Germ(K,U)\to \Germ(K,V)$, $[f]\mto [\phi\circ f]$ is
a homeomorphism, whence $\Germ(K,U)$ is separable.
Recall from Lemma~\ref{folklore}\,(b) that $C(K,G)$ is a separable
topological group.
Now
\[
\alpha\colon \Germ(K,G)\to C(K,G),\quad [f]\mto f|_K
\]
is a continuous homomorphism and $C(K,U)$ an open
identity neighbourhood in $C(K,G)$ such that
$\alpha^{-1}(C(K,U))=\Germ(K,U)$ is separable.
Thus $\Germ(K,G)$ is separable and $k_\omega$,
by Lemma~\ref{givesko}.
\end{proof}
{\bf Proof of Proposition~\ref{arekomeg}.}
Let $M$ be a compact real analytic manifold.
If $G$ is a $\sigma$-compact, finite-dimensional Lie group,
taking $K:=M$ and $\K:=\R$ in Proposition~\ref{germvers},
we see that $C^\omega(M,G)=\Germ(M,G)$ is a $k_\omega$-group.\\[2.3mm]
The Lie group $\Diff^\omega(M)$ is an open subset of the set
$C^\omega(M,M)$ of all real analytic self-maps of~$M$,
endowed with its natural smooth manifold structure (and topology);
see \cite[Theorem 43.3]{KaM} or \cite[Proposition 1.9]{DaS}.
It is useful to recall an aspect of the construction of the manifold
structure on $C^\omega(M,M)$, as described in \cite{KaM} and \cite{DaS}.
Write $0_x$ for the zero vector in $T_xM$ for $x\in M$,
and $\pi_{TM}\colon TM\to M$ for the bundle projection
taking $v\in T_xM$ to~$x$.
Let $\Sigma\colon \Omega\to M$
be a real analytic local addition;
thus $\Omega$ is an open neighbourhood of $\{0_x\colon x\in M\}$
in $TM$, we have
$\Sigma(0_x)=x$
for all $x\in M$, and
\[
(\pi_{TM},\Sigma)\colon \Omega\to M\times M
\]
is a $C^\omega$-diffeomorphism onto an open subset of $M\times M$.
According to \cite[Theorem~1.6\,(a)]{DaS},
the set
\[
U_{\id_M}:=\{\psi\in C^\omega(M,M)\colon (\id_M,\psi)(M)\sub (\pi_{TM},\Sigma)(\Omega)\}
\]
is open in $C^\omega(M,M)$ and the map
\[
\Phi_{\id_M}\colon U_{\id_M}\to \Gamma^\omega(TM),\quad \psi\mto (\pi_{TM},\Sigma)^{-1}\circ (\id_M,\psi)
\]
is a homeomorphism onto an open subset of the Silva space $\Gamma^\omega(TM)$
of real analytic vector fields on~$M$.
Note that
\[
\{(\phi,\psi)\in C(M,M\times M)\colon (\phi,\psi)(M)\sub (\pi_{TM},\Sigma)(\Omega)\}
\]
is open in $C(M,M\times M)$. As
$C(M,M)\to C(M,M\times M)$, $\psi\mto(\id_M,\psi)$
is a continuous map, we deduce that
\[
W:=\{\psi\in\Homeo(M)\colon (\id_M,\psi)(M)\sub (\pi_{TM},\Sigma)(\Omega)\}
\]
is an open $\id_M$-neighbourhood in $\Homeo(M)$.
The inclusion map\linebreak
$\alpha\colon \Diff^\omega(M)\to\Homeo(M)$, $\phi\mto\phi$,
is a continuous homomorphism.
Now
\begin{eqnarray*}
\alpha^{-1}(W)&=&W\cap\Diff^\omega(M)
=\{\psi\in\Diff^\omega(M)\colon (\id_M,\psi)(M)\sub (\pi_{TM},\Sigma)(\Omega)\}\\
&=&\Diff^\omega(M)\cap U_{\id_M}
\end{eqnarray*}
is an open identity neighbourhood in $\Diff^\omega(M)$.
As the homeomorphism $\Phi_{\id_M}$ takes $\alpha^{-1}(W)$ onto an open subset of the Silva space
$\Gamma^\omega(TM)$, we deduce from Lemmas \ref{themodsep} and \ref{easysep}\,(a)
that $\alpha^{-1}(W)$ is separable.
Since $\Homeo(M)$ is separable and metrizable by Lemma~\ref{folklore}\,(c),
we deduce with Lemma~\ref{givesko} that $\Diff^\omega(M)$ is separable
and hence $k_\omega$, by Lemma~\ref{denso}.\,\Punkt\\[2.3mm]
In the next proposition, $\K\in\{\R,\C\}$.
\begin{prop}\label{diffgvers}
For each $n\in\N$ and non-empty compact set $K\sub\K^n=:E$, the Lie group
$\GermDiff(K)$ is separable and a $k_\omega$-group.
\end{prop}
\begin{proof}
As shown in \cite{JFA},
$\Germ(K,E)_0:=\{[f]\in \Germ(K,E)\colon f|_K=0\}$
is a Silva space,
$\Omega:=\{ [f]\in\Germ(K,E)_0\colon [\id_E+f]\in\GermDiff(K)\}$
is open in $\Germ(K,E)_0$ and the bijection
$\Omega\to\GermDiff(K)$,
$[f]\mto [\id_E+f]$
is a homeomorphism for the Lie group
structure on
$\GermDiff(K)$. Since~$\Omega$ is separable by Lemma~\ref{themodsep}
and \ref{baby-metric}, also $\GermDiff(K)$ is separable
and hence~$k_\omega$, by Proposition~\ref{modelloc} and Lemma~\ref{denso}.
\end{proof}
{\small
{\bf Helge  Gl\"{o}ckner}, Institut f\"{u}r Mathematik,
Universit\"at Paderborn,\\
Warburger Str.\ 100, 33098 Paderborn, Germany; {\tt glockner@math.upb.de}}\vfill
\end{document}